\documentclass[12pt]{amsart}
\usepackage{a4wide,enumerate,xcolor,graphicx}
\usepackage{amsmath,comment,algorithm}
\allowdisplaybreaks
\usepackage{tabularx}
\usepackage{multirow}
\usepackage{graphicx}

\let\pa\partial
\let\na\nabla
\let\eps\varepsilon

\newcommand{\R}{{\mathbb R}}
\newcommand{\dd}{\mathrm{d}}
\newcommand{\diver}{\operatorname{div}}

\newtheorem{theorem}{Theorem}

\newtheorem{definition}[theorem]{Definition}
\newtheorem{assumptions}[theorem]{Assumption}

\begin{document}

\title[Bayesian inversion]{Bayesian inversion for the identification of the doping profile in unipolar semiconductor devices}

\author[L. Taghizadeh]{Leila Taghizadeh}
\address{Institute of Analysis and Scientific Computing, TU Wien, Wiedner Hauptstra\ss e 8--10, 1040 Wien, Austria}
\email{leila.taghizadeh@tuwien.ac.at}

\author[A. J\"ungel]{Ansgar J\"ungel}
\address{Institute of Analysis and Scientific Computing, TU Wien, Wiedner Hauptstra\ss e 8--10, 1040 Wien, Austria}
\email{juengel@tuwien.ac.at} 

\date{\today}

\thanks{The first author has been supported by an Elise-Richter grant, funded by the Austrian Science Fund (FWF), grant DOI 10.55776/V1000. The last author acknowledges partial support from the FWF, grant DOI 10.55776/P33010 and 10.55776/F65. This work has received funding from the European Research Council (ERC) under the European Union's Horizon 2020 research and innovation programme, ERC Advanced Grant NEUROMORPH, no.~101018153}. 

\begin{abstract}
A rigorous Bayesian formulation of the inverse doping profile problem in infinite dimensions for a stationary linearized unipolar drift-diffusion model for semiconductor devices is given. The goal is to estimate the posterior probability distribution of the doping profile and to compute its posterior mean. This allows for the reconstruction of the doping profile from voltage--current measurements. The well-posedness of the Bayesian inverse problem is shown by proving boundedness and continuity properties of the semiconductor model with respect to the unknown parameter. A preconditioned Crank--Nicolson Markov chain Monte--Carlo method for the Bayesian estimation of the doping profile, using a physics-informed prior model, is proposed. The numerical results for a two-dimensional diode illustrate the efficiency of the proposed approach.
\end{abstract}

 

\maketitle


\section{Introduction}

In semiconductor manufactoring, semiconductor crystals are typically doped by impurities to modulate the electrical and structural properties of the crystal. Doping can be achieved during the crystal growth via vapor-phase epitaxy or later by diffusion and ion implantation techniques \cite[Chap.~10]{MiHo92}. Since the doping concentration is not always uniform across the thickness, which may lead to device failures, it is important to determine the doping profile inside the semiconductor. A noninvasive method to estimate doping inhomogeneities is the identification by voltage--current measurements \cite{KFBS95}. The inverse problem to reconstruct the doping profile from measurements of the current density at an Ohmic contact is highly ill-posed because of the limited number of data points. In the literature, the semiconductor inverse problem has been solved by deterministic approaches only, using drift-diffusion equations \cite{BELM04,BEMP01,Lei06}, Boltzmann--Poisson models \cite{CGR11}, or more recently, data-driven methods \cite{PFLRH24}. In this paper, we investigate the doping profile inverse problem from a Bayesian perspective for the first time (up to our knowledge). 

Bayesian inversion is a statistical inversion approach, which allows one to incorporate the uncertainties in the measurement data to estimate the probability density of the unknown parameter, given prior knowledge and observation data. Bayesian estimation has been used, for instance, for nanowire field-effect sensors \cite{KSH20}, biofilm growth modeling \cite{TKPH20}, electrical impedance tomography in medical imaging \cite{KTH21}, dynamic spectroscopic current imaging \cite{SLM+18}, and time-resolved photoluminescence \cite{FHL23}. 

Because of the high complexity of the full bipolar drift-diffusion system, we consider a simplified model, the stationary linearized unipolar drift-diffusion equations for semiconductors (close to equilibrium) as in \cite{BELM04,BEMP01}. The linearization is motivated by the fact that stationary solutions may be nonunique for large applied voltages (far from equilibrium). The model consists of two decoupled elliptic problems, a continuity equation for the Slotboom variable for electrons and the Poisson equation for the electric potential. The corresponding inverse problem is split into two PDE-governed inverse problems. The difficulties are the high dimensionality of the parameter space, the ill-posednes of the inverse problem, and the limited measured data from the surface electrode. 

To overcome these issues, we perform a preconditioned Crank--Nicolson Markov chain Monte--Carlo (MCMC) estimation of the posterior distribution of the unknown parameter in an infinite-dimensional setting. This method produces random samples from a target probability distribution for which direct sampling is difficult, and it is well-suited for high-dimensional sampling problems. We use a physics-informed prior (from the Poisson equation) to enhance the prior knowledge of the unknown parameter, to compensate for the data limit, and hence to enhance the posterior estimation and the resulting Bayesian reconstruction of the doping profile.

The paper is organized as follows. In Section \ref{sec.model}, we introduce the mathematical semiconductor model, specify our simplifying
assumptions that lead to the forward model, and introduce the semiconductor inverse problem. The Bayesian inverse problem is formulated in Section \ref{sec.bayes}, where we also show the well-posedness of the Bayesian posterior. The proposed preconditioned Crank--Nicolson MCMC algorithm is detailed in Section \ref{sec.MCMC}, and numerical results are reported in Section \ref{sec.num}. Finally, conclusions are drawn in Section \ref{sec.conc}. 


\section{Forward model}\label{sec.model}

In this section, we introduce the drift-diffusion equations for semiconductors, detail our simplifying assumptions, and formulate the semiconductor inverse problem. 

\subsection{Drift-diffusion equations}

The semiconductor drift-diffusion equations for are wide\-ly used to describe the charge transport of electrons and holes (defect electrons) in semiconductor devices \cite{Mar86,Sel84}. The equations can be derived from the semiconductor Boltzmann equation in the diffusion limit \cite{Jue09}. The stationary equations for the electron density $n$, the hole density $p$, and the electric potential $V$ are given by 
\begin{equation}\label{2.biDD}
\begin{aligned}
  & \diver J_n = qR(n,p), \quad J_n = q(D_n\na n - \mu_n n\na V), \\
  & \diver J_p = -qR(n,p), \quad J_p = -q(D_p\na p + \mu_p p\na V), \\
  & \diver(\eps_s V) = q(n-p-C(x))\quad\mbox{in }\Omega,
\end{aligned}
\end{equation}
where $\Omega\subset\R^d$ ($d\ge 1$) is the physical domain and the physical parameters are the elementary charge $q$, the electron and hole diffusion coefficients $D_n$ and $D_p$, the electron and holes mobilities $\mu_n$ and $\mu_p$, respectively, and the semiconductor permittivity $\eps_s$. The recombination--generation term reads as $R(n,p) = R_0(n,p)(np-n_i^2)$, where $R_0(n,p)=(\tau_n(n+n_i)+\tau_n(p+n_i))^{-1}$ is the Shockley--Read--Hall model with the intrinsic density $n_i$ and the constants $\tau_n$, $\tau_p>0$. The current densities $J_n$ and $J_p$ consist of the diffusion current densities $qD_n\na n$, $-qD_p\na p$ and the drift densities $-q\mu_n n\na V$, $-q\mu_p p\na V$, respectively. We suppose that the diffusivities and mobilities are related by the Einstein relations $D_{n/p}=U_T\mu_{n/p}$, where $U_T=k_B\theta_L/q$ is the thermal voltage, $k_B$ the Boltzmann constant, and $\theta_L$ the lattice temperature. The doping profile $C(x)$ denotes the given concentration of dopant atoms in the device. 

We assume that the boundary $\pa\Omega=\pa\Omega_D\cup\pa\Omega_N$ is partitioned into two parts, the union of Ohmic contacts $\pa\Omega_D$ and the insulated parts $\pa\Omega_N$. The boundary conditions read as
\begin{equation}\label{2.bc}
\begin{aligned}
  n = n_D, \quad p = p_D, \quad V = V_D &\quad\mbox{on }\pa\Omega_D, \\
  \na n\cdot\nu = \na p\cdot\nu = \na V\cdot\nu = 0 
  &\quad\mbox{on }\pa\Omega_N,
\end{aligned}
\end{equation}
where the Dirichlet data are determined from the assumptions that the total space charge vanishes on $\pa\Omega_D$ and the densities are in thermal equilibrium on $\Gamma_D$ \cite[Sec.~5.3]{Jue09},
\begin{align*}
  n_D = \frac12\Big(C+\sqrt{C^2+4n_i^2}\Big), \
  p_D = \frac12\Big(-C+\sqrt{C^2+4n_i^2}\Big), \
  V_D = V_{\text{bi}} + U\quad\mbox{on }\pa\Omega_D.
\end{align*}
Here, $U$ is the applied potential and $V_{\text{bi}}=U_T\ln(n_D/n_i)=U_T\operatorname{arsinh}(C/(2n_i))$ the built-in potential.

In this paper, we restrict ourselves to a simple $pn$-diode. The Dirichlet boundary consists of two Ohmic contacts $\Gamma_N$ and $\Gamma_P$, i.e.\ $\pa\Omega_D=\Gamma_P\cup\Gamma_N$; see Figure \ref{fig.diode}. Here, the doping profile is supposed to be piecewise constant according to
\begin{align*}
  C(x) = \begin{cases}
  C_N &\quad\mbox{for }x\in\Omega_N, \\
  -C_P &\quad\mbox{for }x\in\Omega_P,
  \end{cases}
\end{align*}
where $C_N$, $C_P>0$ are two constants, $\Omega_N$ is called the $n$-region, and $\Omega_P$ the $p$-region. 

\begin{figure}[ht]
\centering
\includegraphics[width=90mm]{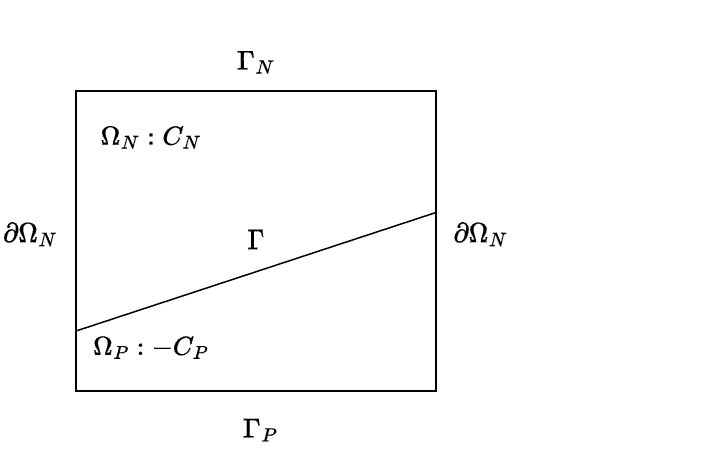}
\caption{Cross-section of a $pn$-diode with two Ohmic contacts $\Gamma_N$ and $\Gamma_P$. The junction $\Gamma$ separates the $n$- and $p$-regions.}
\label{fig.diode}
\end{figure}

Scaling the equations as in \cite[Sec.~2.4]{Mar86} and introducing (after the scaling) the Slotboom variables $u=\delta^{-2}\exp(-V)n$, $v=\delta^{-2}\exp(V)p$, system \eqref{2.biDD} becomes
\begin{equation}\label{2.biDDs}
\begin{aligned}
  & \diver J_n = \delta^4 Q(V,u,v)(uv-1), \quad
  J_n = \mu_n\delta^2 e^V\na u, \\
  & \diver J_p = -\delta^4 Q(V,u,v)(uv-1), \quad
  J_p = -\mu_p\delta^2 e^{-V}\na v, \\
  & \lambda^2\Delta V = \delta^2(e^V u - e^{-V}v) - C(x)
  \quad\mbox{in }\Omega,
\end{aligned}
\end{equation}
where $\delta^2=n_i/\max_\Omega |C|$ is the scaled intrinsic number and
$Q(V,u,v)=R_0(n,p)$. The scaled boundary data is given by
\begin{equation}\label{2.bcs}
\begin{aligned}
  n_D = \frac12\Big(C + \sqrt{C^2+4\delta^4}\Big), \quad
  p_D = \frac12\Big(-C + \sqrt{C^2+4\delta^4}\Big), \quad
  V_D = V_{\text{bi}} + U,
\end{aligned}
\end{equation}
where $V_{\text{bi}}=\operatorname{arsinh}(C/(2\delta^2))$ is the scaled built-in potential. Under suitable assumptions on the data, there exists a weak solution to system \eqref{2.bc}, \eqref{2.biDDs}, and \eqref{2.bcs} \cite[Theorem 3.2.1]{Mar86}. The uniqueness of weak solutions is more delicate, since certain semiconductor devices admit multiple stationary states for sufficiently large applied voltages. However, it is known that the weak solution is unique if the applied voltage is sufficiently small in the $L^\infty(\pa\Omega_D)\cap H^1(\pa\Omega_D)$ norm \cite[Theorem 2.4]{BEMP01}. 

A special steady state is given by the thermal equilibrium, which is is defined by $uv=1$ and $U=0$. Then $u=v=1$ \cite[Sec.~5.3]{Jue09} and consequently $J_n=J_p=0$ as well as
\begin{align*}
  & \lambda^2\Delta V_e = \delta^2(e^{V_e} - e^{-V_e}) - C(x)
  \quad\mbox{in }\Omega, \\
  & V_e = V_{\text{bi}}\quad\mbox{on }\pa\Omega_D, \quad
  \na V_e\cdot\nu = 0 \quad\mbox{on }\pa\Omega_N.
\end{align*}
Since the forward model is uniquely solvable only for sufficiently small $U$, we consider the linearized model around the equilibrium state. To simplify further, we also assume that there are no recombination--generation effects ($Q=0$) and that the main charge carrier are electrons ($p=0$). This leads to the unipolar model
\begin{align*}
  & \diver(e^V\na u) = 0\quad\mbox{in }\Omega, \quad u=u_D
  \quad\mbox{on }\pa\Omega_D, \quad 
  \na u\cdot\nu = 0 \quad\mbox{on }\pa\Omega_N, \\
  & \lambda^2\Delta V = \delta^2 e^V u - C(x)
  \quad\mbox{in }\Omega, \quad V = V_{\text{bi}} + U 
  \quad\mbox{on }\pa\Omega_D, \quad
  \na V\cdot\nu = 0 \quad\mbox{on }\pa\Omega_N,
\end{align*}
where $u_D=\delta^{-2}\exp(-V_D)n_D=\exp(-U)$. The linearized unipolar drift-diffusion model (close to equilibrium) is obtained by computing the variational derivative with respect to $U$ at $U=0$ in the direction of $h$,
\begin{align}\label{2.u1}
  & \diver\widehat{J}_n = 0, \quad  
  \widehat{J}_n = \mu_n\delta^2 e^{V_e}\na\widehat{u}
  \quad\mbox{in }\Omega, \\
  & \widehat{u} = -h \quad\mbox{on }\pa\Omega_D, \quad
  \na\widehat{u}\cdot\nu = 0 \quad\mbox{on }\pa\Omega_N, \label{2.u2}
\end{align}
where $V_e$ is the unique solution to the Poisson equation at equilibrium,
\begin{align}\label{2.Ve}
  \lambda^2\Delta V_e = \delta^2 e^{V_e} - C(x)\quad\mbox{in }\Omega, \quad  V_e = V_{\text{bi}}\quad\mbox{on }\pa\Omega_D, \quad
  \na V_e\cdot\nu = 0 \quad\mbox{on }\pa\Omega_N. 
\end{align}
We call system \eqref{2.u1}--\eqref{2.Ve} the {\em semiconductor forward problem}. Problem \eqref{2.Ve} possesses a unique bounded weak solution (if $V_{\text{bi}}$ is bounded). Then problem \eqref{2.u1}--\eqref{2.u2} is also uniquely solvable and $\widehat{u}\in H^1(\Omega)$. 

For the inverse problem, we need to compute the normal trace $\widehat{J}_n\cdot\nu$ at the contact $\Gamma_N$. Generally, the normal trace is defined as an element of the dual of the Lions--Magenes space $H_{00}^{1/2}(\Gamma_N)$, where $u\in H_{00}^{1/2}(\Gamma_N)$ if and only if the trivial extension of $u$ to $\pa\Omega$ belongs to $H^{1/2}(\pa\Omega)$ \cite[Chap.~18]{BaCa84}. As in \cite{BELM04,BEMP01,Lei06}, we prefer to work with normal traces that are defined a.e., since these are the measured real-valued data. Indeed, if the Dirichlet boundary data lies in the space $H^{3/2}(\pa\Omega_D)$ and the geometry of the domain is as in Figure \ref{fig.diode}, the solution to a linear elliptic problem with mixed Dirichlet--Neumann boundary conditions has the maximal regularity $H^2(\Omega)$. In particular, $V_e\in H^2(\Omega)\hookrightarrow L^\infty(\Omega)$, $\widehat{u}\in H^2(\Omega)$, and consequently, $\widehat{J}_n\in H^1(\Omega)$. In this situation, the normal trace $\widehat{J}_n\cdot\nu$ on $\Gamma_N\subset\pa\Omega_D$ is an element of $H^{1/2}(\Gamma_N)$ and exists a.e.


\subsection{Inverse problem}

We introduce the stationary voltage--current map 
\begin{align*}
  \Sigma_{\text{VC}}: H^{3/2}(\pa\Omega_D)\to H^{1/2}(\Gamma_N), 
  \quad U\mapsto \widehat{J}_n\cdot\nu|_{\Gamma_N},
\end{align*}
where $\widehat{J}_n$ is given by \eqref{2.u1} and $\Gamma_N\subset\pa\Omega_D$ is a Dirichlet contact (see Figure \ref{fig.diode}). We deduce from the discussion of the previous subsection that the mapping $\Sigma_{\text{VC}}$ is well defined. Note that we require pointwise measurements at the contact $\Gamma_N$. An alternative approach is given by current-flow measurements from $\int_{\Gamma_N}\widehat{J}_n\cdot\nu ds$. They are easier to realize in practice but they yield much less information about the doping profile than pointwise measurements \cite[Sec.~3]{BELM04}. 

The G\^ateaux derivative of $\Sigma_{\text{VC}}$ at $U=0$ in the direction $h\in H^{3/2}(\pa\Omega_D)$ exists and reads as
\begin{align*}
   \Sigma'_{\text{VC}}(0)h 
  = \mu_n\delta^2e^{V_{\text{bi}}}\na\widehat{u}\cdot\nu\big|_{\Gamma_N},
\end{align*}
where $\widehat{u}$ solves \eqref{2.u1}--\eqref{2.u2}; see \cite[Prop.~3.1]{BEMP01} for a proof. The inverse problem of identifying the doping profile, given the current measurements, corresponds to the identification of the doping profile from the parameter-to-observable map
\begin{align*}
  G: \operatorname{dom}(G)\to \mathcal{L}(H^{3/2}(\pa\Omega_D);H^{1/2}(\Gamma_N)), 
  \quad C\mapsto \Sigma'_{\text{VC}}(0),
\end{align*}
where the domain of $G$ is given by
\begin{align*}
  \operatorname{dom}(G) = \big\{C\in L^2(\Omega):\underline{C}\le C\le\overline{C}\mbox{ a.e.\ in }\Omega\big\}
\end{align*}
and $\underline{C}$, $\overline{C}$ are positive constants. We summarize the inverse problem in Algorithm \ref{algo1}; also see \cite[Sec.~4]{BELM04}.

\begin{algorithm}[th!]
\caption{Inverse doping profile problem with stationary linearized unipolar model}\label{algo1}
\begin{enumerate}
\item[(i)] Define $\gamma(x):=\exp(V_e(x))$ for $x\in\Omega$.
\item[(ii)] Solve the parameter identification problem
\begin{align}\label{2.gamma}
  \diver(\gamma(x)\na\widehat{u}) = 0 \quad\mbox{in }\Omega, \quad
  \widehat{u} = U\quad\mbox{on }\pa\Omega_D, \quad  
  \na\widehat{u}\cdot\nu = 0 \quad\mbox{on }\pa\Omega_N,
\end{align}
given the measurements of $\Lambda_\gamma(U) :=\Sigma'_{\text{VC}}(0)U = \mu_n\delta^2e^{V_{\text{bi}}}\na\widehat{u}\cdot\nu|_{\Gamma}$.
\item[(iii)] Determine the doping profile from 
\begin{equation}\label{2.C}
  C(x)=\gamma(x)-\lambda^2\Delta(\ln\gamma)(x) \quad\mbox{for } x\in\Omega.
\end{equation} 
\end{enumerate}
\end{algorithm}

The inverse problem of identifying the doping profile in problem \eqref{2.u1}--\eqref{2.Ve} from the measurements of $\widehat{J}_n\cdot\nu$ at the contact $\Gamma_N$ reduces to the problem of identifying the parameter $\gamma(x)$ in \eqref{2.gamma} from measurements of the voltage-to-current map $\Sigma_{\text{VC}}$. This problem is well known as the conductivity inverse problem, which is highly ill-posed. In Figure \ref{fig.flowchart}, the forward and inverse problems are illustrated schematically. For the measurements, we have a finite number of measurements to identify $\gamma$, $(U_j,\Lambda_\gamma(U_j))_{j=1}^m$, where the applied voltages $U_j$ are assumed to be piecewise constant on the contacts, vanish on $\pa\Omega_N$, and are small in some norm.

\begin{figure}[ht]
\includegraphics[width=170mm]{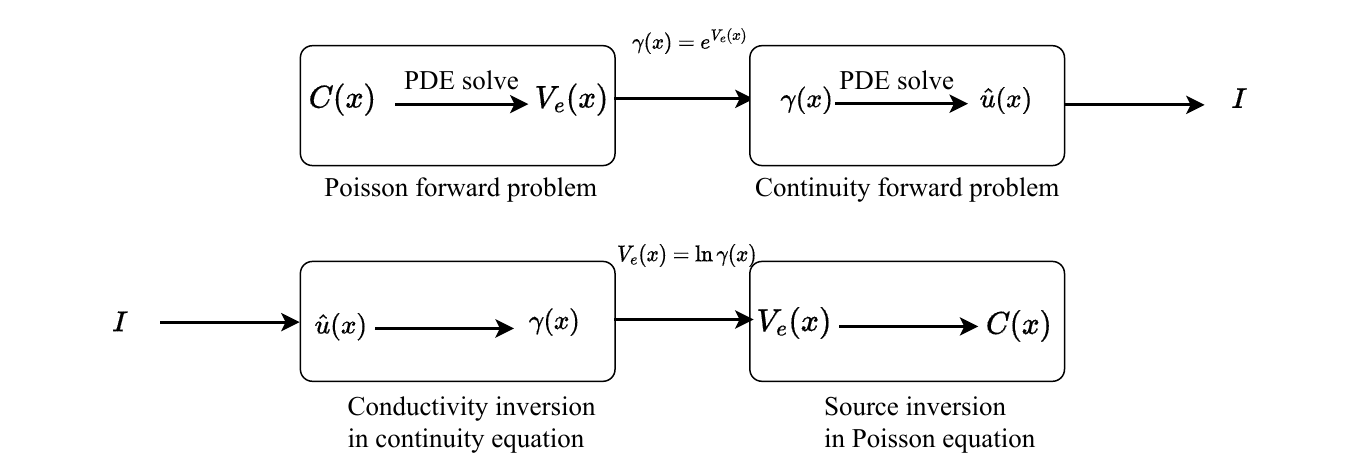}
\caption{Schematic diagrams of the forward (top) and inverse (bottom) problems for semiconductors with $I=\widehat{J}_n\cdot\nu|_{\Gamma_N}$.}
\label{fig.flowchart}
\end{figure}


\section{Semiconductor Bayesian inverse problem}\label{sec.bayes}

We review briefly inverse problems in the framework of statistical inversion theory and formulate the semiconductor inverse problem as a Bayesian inverse problem with the goal to recover the infinite-dimensional doping profile from the current density measurements, i.e., we first recover $\gamma$ from the inverse problem \eqref{2.gamma} and then compute the doping profile explicitly from \eqref{2.C}. 

To set up the Bayesian formulation, let $y=(y_1,\ldots,y_m)$ with $y_j=\Lambda_\gamma(U_j)$ for $j=1,\ldots,m$ be an observation vector of noisy pointwise data in the finite-dimensional observation space $Y:=\R^m$, let $\eta=(\eta_1,\ldots,\eta_m)$ be the vector of additive noise, and let $V_e$ be a parameter field in the parameter space $X:=L^2(\Omega)$ associated with the forward model \eqref{2.u1}--\eqref{2.Ve}. The observation model is defined by
\begin{align}\label{3.obmodel}
  y = G(V_e) + \eta, \quad \eta\sim\mathcal{N}(0,\Sigma),
\end{align}
where $\eta$ is a Gaussian noise with zero mean and (positive definite) covariance matrix $\Sigma\in\R^{m\times m}$. Furthermore, $G:X\to Y$, $G(V_e)=y$, is the so-called parameter-to-observable (forward) map, which is the composition of the parameter-to-state map $V_e\mapsto\widehat{u}$ and the state-to-observation map $\widehat{u}\mapsto y$. 

We recall that, in the semiconductor problem, the parameter-to-observable map $G$ calculates the normal trace $\widehat{J}_n\cdot\nu$ at certain points of the contact $\Gamma_N$, given the doping profile, and the noisy data $y$ is the same normal trace of the current density plus some Gaussian noise with zero mean and covariance $\Sigma$. 

In the Bayesian approach, we assume that the unknown $V_e$ is a random function (which becomes a parameter vector after discretization), and the observation vector $y$ and the observational noise vector $\eta$ are random variables. We assign a prior probability measure to the unknown parameter, expressing the a priori information. Then the overall uncertainty is reduced by conditioning the prior on the observational data, leading to the posterior measure. Therefore, the goal is to find a conditional posterior probability measure on the unknown parameter, given the data (likelihood), and a prior probability measure on the unknown parameter. We detail these measures in the next subsection.

\subsection{Likelihood, prior, and posterior distribution}\label{sec.lpp}

Since the observational noise is Gaussian, the negative log-likelihood function is defined by
\begin{align}\label{3.Phi}
  \Phi: X\times Y\to\R, \quad \Phi(V_e;y) = \frac12\|y-G(V_e)\|_\Sigma^2
  := \frac12\|\Sigma^{-1/2}(y-G(V_e))\|_2^2,
\end{align}
where we used the density of the noise with respect to the Lebesgue measure in $Y$ and $\|\cdot\|_2$ is the Euclidean norm. Hence, the likelihood function is given by
\begin{align}\label{3.like}
  \pi(y|V_e) = \exp\bigg(-\frac12\|y-G(V_e)\|_\Sigma^2\bigg).
\end{align}
It is well known \cite[Chap.~10, p.~140]{Van00} that, given a high-dimensional measurement space $Y=\R^m$ (i.e., $m$ is large enough) and a small measurement noise and assuming that the prior density is smooth and positive in a neighborhood of the true parameter values, the choice of the prior measure does not have a significant impact on the posterior measure, since the likelihood is informative enough. Unfortunately, this is not the case for the semiconductor inverse problem in the sense that the measurement data is not sufficiently large as they are collected from the surface electrode only. In fact, the number of data points is limited and our inverse problem is highly ill-posed. Thus, the choice of the prior is important, as it is expected to have a significant impact on the posterior distribution.

We use a Gaussian random field as the prior model. More precisely, we choose
\begin{align}\label{3.mu0}
  \mu_0 = \mathcal{N}(m_0,\Sigma_0),
\end{align}
where $m_0$ is the mean and $\Sigma_0:X\to X$ is the prior covariance, which is a symmetric, nonnegative, and trace-class operator. For details about $m_0$ and $\Sigma_0$, we refer to \cite{IbRo78}. Here, we take the Mat\'ern--Whittle covariance operator
\begin{align}\label{3.MW}
  \Sigma_0(x,y) = \sigma^2\frac{2^{1-\nu}}{\Gamma(\nu)}
  \bigg(\frac{d_0}{\ell}\bigg)^\nu K_\nu\bigg(\frac{d_0}{\ell}\bigg),
  \quad x,y\in\Omega,
\end{align}
where $d_0=\|x-y\|_2$, $\Gamma$ is the Gamma function, $K_\nu$ is the modified Bessel function, and the parameters $\sigma^2$, $\nu$, and $\ell$ denote the variance, smoothness, and spatial correlation length, respectively. The Mat\'ern--Whittle operator allows us to control the amplitude, smoothness, and correlation length of the generated samples.  We refer to Section \ref{sec.prior} for the discretization of the prior random field. 

We denote by $\pi_0(V_e)$ the prior density associated to the prior measure $\mu_0$, given by \eqref{3.mu0}, and recall the data likelihood $\pi(y|V_e)$, defined in \eqref{3.like}. If $V_e$ is from a finite-dimensional space, the posterior density $\pi(V_e|y)$, associated to the posterior measure $\mu^y$, is obtained from the Bayes rule
\begin{align*}
  \pi(V_e|y) \propto \pi(y|V_e)\pi_0(V_e).
\end{align*}
However, for infinite-dimensional spaces $X$ (as in our case), there is no density with respect to the Lebesgue measure. Therefore, the Bayes rule should be interpreted by means of the Radon--Nikod\'ym derivative
\cite[Sec.~2]{Stu10}
\begin{align}\label{3.bayes}
  \frac{\dd\mu^y}{\dd\mu_0}(V_e) = \frac{\exp(-\Phi(V_e;y))}{Z(y)},
  \quad Z(y) = \int_X\exp(-\Phi(V_e;y))\dd\mu_0(V_e).
\end{align}
The well-posedness of this problem is discussed next.


\subsection{Well-posedness of the Bayesian inverse problem}

We prove the well-posedness of the Bayesian inverse problem \eqref{2.gamma}--\eqref{2.C}. For this, we introduce the Hellinger distance, investigate the well-posedness of the posterior measure in the Hellinger distance, and define the well-posedness of the Bayesian inverse problem in the general setting before verifying the hypotheses for the semiconductor problem at hand.

Let $(X,\mathcal{A},\mu)$ be a measure space, consisting of a set $X$, a $\sigma$-algebra $\mathcal{A}$ on $X$, and a measure $\mu$ defined on $\mathcal{A}$. We introduce the space 
\begin{align*}
  \operatorname{meas}(X,\nu)=\{\mu\mbox{ measure on }X:\mu\ll\nu\},
\end{align*}
where $\mu\ll\nu$ means that $\mu$ is absolutely continuous with respect to $\nu$. The {\em Hellinger distance} between two probability measures $\mu_1$, $\mu_2\in\operatorname{meas}(X,\nu)$ is given by
\begin{align*}
  H(\mu_1,\mu_2) = \bigg\{\frac12\int_X\bigg(
  \sqrt{\frac{\dd\mu_1}{\dd\nu}}
  - \sqrt{\frac{\dd\mu_2}{\dd\nu}}\bigg)^2\dd\nu\bigg\}^{1/2}.
\end{align*}

\begin{definition}[Well-posedness in the Hellinger distance]\label{def.well}
Let $X$ be a Banach space. The Bayesian inverse problem \eqref{3.bayes}, for any prior measure $\mu_0$ and any log-likelihood function $\Phi$, is well-posed if the following two properties hold:
\begin{enumerate}
\item[\rm (i)] (Well-posedness) There exists a unique posterior probability measure $\mu^y\in\operatorname{meas}(X,$ $\mu_0)$ given by \eqref{3.bayes}.
\item[\rm (ii)] (Stability) The posterior measure $\mu^y$ is locally Lipschitz continuous with respect to $y$, i.e., for any $r>0$, there exists $C(r)>0$ such that for all $y_1$, $y_2\in Y$ with $\max\{\|y_1\|_Y,\|y_2\|_Y\}<r$, it holds that
\begin{align*}
  H(\mu^{y_1},\mu^{y_2}) \le C(r)\|y_1-y_2\|_Y.
\end{align*}
\end{enumerate}
\end{definition}

We recall the assumptions on the log-likelihood function $\Phi$ so that the corresponding Bayesian inverse problem is well-posed \cite[Assumption 2.6]{Stu10}. 

\begin{assumptions}\label{assump}
Let $X$ and $Y$ be Banach spaces. We suppose that $\Phi:X\times Y\to\R$ has the following properties:
\begin{enumerate}
\item[\rm (i)] (Lower bound) For any $\alpha>0$ and $r>0$, there exists $M=M(\alpha,r)\in\R$ such that for all $V_e\in X$ and $y\in Y$ with $\|y\|_Y<r$,
\begin{align*}
  \Phi(V_e;y) \ge M-\alpha\|V_e\|_X.
\end{align*} 
\item[\rm (ii)] (Upper bound) For any $r>0$, there exists $K(r)>0$ such that for all $V_e\in X$ and $y\in Y$ with $\max\{\|V_e\|_X,\|y\|_Y\}<r$,
\begin{align*}
  \Phi(V_e;y) \le K(r).
\end{align*}
\item[\rm (iii)] (Continuity in $y$) For any $r>0$, there exists $L(r)>0$ such that for all $V_e^1,V_e^2\in X$ and $y_1,y_2\in Y$ with $\max\{\|V_e^1\|_X,\|V_e^2\|_X,\|y\|_Y\}<r$,
\begin{align*}
  |\Phi(V_e^1;y)-\Phi(V_e^2;y)|\le L(r)\|V_e^1-V_e^2\|_X.
\end{align*}
\item[\rm (iv)] (Continuity in $y$) For any $\beta,r>0$, there exists $C=C(\beta,r)>0$ such that for all $V_e\in X$ and $y_1,y_2\in Y$ with $\max\{\|y_1\|_Y,\|y_2\|_Y\}<r$,
\begin{align*}
  |\Phi(V_e;y_1)-\Phi(V_e;y_2)|\le \exp\big(\beta\|V_e\|_X+C\big)\|y_1-y_2\|_Y.
\end{align*}
\end{enumerate}
\end{assumptions}

The following theorem, taken from \cite[Theorems 4.1--4.2]{Stu10}, shows that if Assumption \ref{assump} are satisfied, the Bayesian inverse problem is well-posed in the sense of Definition \ref{def.well}.

\begin{theorem}[Well-posedness and stability]
Let $X$ and $Y$ be Banach spaces and let $\mu_0$ be a Gaussian measure with $\mu_0(X)=1$. If $\Phi$ satisfies Assumption \ref{assump}, then there exists a unique posterior measure $\mu^y$ for \eqref{3.bayes} and $\mu^y$ is Lipschitz continuous with respect to $y$ in the Hellinger distance. 
\end{theorem}

If the observation space $Y$ is finite-dimensional and we assume a Gaussian noise in the observational model \eqref{3.obmodel}, the log-likelihood function is of the form \eqref{3.Phi}. In this case, which is our case as well (since $Y=\R^m$), the lower bound in Assumption \ref{assump} is satisfied with $M=\alpha=0$. Therefore, Assumption \ref{assump} can be reduced to the following two assumptions \cite[Theorem 4.1]{HoNi17}.

\begin{theorem}[Reduced assumptions for well-posedness]
Let $X$ be a Banach space with $\mu_0(X)=1$ and let $G:X\to\R^m$, defined in \eqref{3.obmodel}, satisfy:
\begin{enumerate}
\item[\rm (i)] For any $\eps>0$, there exists $M=M(\eps)>0$ such that for all $V_e\in X$,
\begin{align*}
  \|G(V_e)\|_\Sigma \le \exp(M+\eps\|V_e\|_X).
\end{align*}
\item[\rm (ii)] For any $r>0$, there exists $L(r)>0$ such that for all $V_e^1,V_e^2\in X$ with $\max\{\|V_e^1\|_X,$ $\|V_e^2\|_X\}<r$,
\begin{align*}
  \|G(V_e^1)-G(V_e^2)\|_\Sigma\le L(r)\|V_e^1-V_e^2\|_X.
\end{align*}
\end{enumerate}
Then $\Phi$, defined in \eqref{3.Phi}, satisfies Assumptions \ref{assump}. In particular, the Bayesian inverse problem \eqref{3.bayes} is well-posed in the sense of Definition \ref{def.well} with $Y=\R^m$ and $\|\cdot\|_Y=\|\cdot\|_\Sigma$. 
\end{theorem}

As only the continuity equation \eqref{2.u1} is directly involved in the Bayesian inverse problem for semiconductors, the following well-posedness theorem is established on this equation.

\begin{theorem}
Consider the decoupled semiconductor problem \eqref{2.u1}--\eqref{2.Ve} with Dirichlet condition $\widehat{u}_D=U\in H^{3/2}(\pa\Omega_D)$ on $\pa\Omega_D$. Then there exist a constant $C_D>0$ only depending on the Dirichlet boundary data such that
\begin{align*}
  \|\widehat{u}\|_{H^1(\Omega)} 
  \le C_D\exp\big({2}\|V_e\|_{L^\infty(\Omega)}\big).
\end{align*}
Furthermore, let $V_e^1$, $V_e^2$ be two bounded weak solutions to \eqref{2.Ve} and $\widehat{u}_1$, $\widehat{u}_2$ the corresponding weak solutions to \eqref{2.u1}--\eqref{2.u2}. Then
\begin{align*}
  \|\widehat{u}_1-\widehat{u}_2\|_{H^1(\Omega)}
  \le C_D\exp\big({4}\max\{\|V_e^1\|_{L^\infty(\Omega)},
  \|V_e^2\|_{L^\infty(\Omega)}\}\big)\|V_e^1-V_e^2\|_{L^\infty(\Omega)}.
\end{align*}
\end{theorem}

\begin{proof} 
According to \cite[Theorem 4.10 (i)]{McL00}, there exists a Dirichlet lift $\bar{U}\in H^1(\Omega)$ of $U$ satisfying $\Delta\bar{U}=0$ in $\Omega$, $\bar{U}=U$ on $\pa\Omega_D$, and $\na\bar{U}\cdot\nu=0$ on $\pa\Omega_N$. Since $U\in H^{3/2}(\pa\Omega_D)\hookrightarrow L^\infty(\pa\Omega_D)$, we can apply the maximum principle to infer that $\bar{U}\in L^\infty(\Omega)$. Consequently, we can extend $\widehat{u}_D=U\in L^\infty(\Omega)\cap H^1(\Omega)$. We use $\widehat{u}-\widehat{u}_D\in H^1(\Omega)$ as an admissible test function in the weak formulation of \eqref{2.u1}--\eqref{2.u2}:
{\begin{align*}
  0 = \int_\Omega e^{V_e}\na\widehat{u}
  \cdot\na(\widehat{u}-\widehat{u}_D)\dd x 
  = \int_\Omega e^{V_e}|\na(\widehat{u}-\widehat{u}_D)|^2\dd x
  + \int_\Omega e^{V_e}\na\widehat{u}_D
  \cdot\na(\widehat{u}-\widehat{u}_D)\dd x. 
\end{align*}
This shows that
\begin{align*}
  \exp&(-\|V_e\|_{L^\infty(\Omega)})
  \|\na(\widehat{u}-\widehat{u}_D)\|_{L^2(\Omega)}^2
  \le \int_\Omega e^{V_e}|\na(\widehat{u}-\widehat{u}_D)|^2\dd x \\
  &= -\int_\Omega e^{V_e}\na\widehat{u}_D
  \cdot\na(\widehat{u}-\widehat{u}_D)\dd x
  \le \exp(\|V_e\|_{L^\infty(\Omega)})\|\widehat{u}_D\|_{L^2(\Omega)}
  \|\na(\widehat{u}-\widehat{u}_D)\|_{L^2(\Omega)},
\end{align*}
and, after division by $\|\na(\widehat{u}-\widehat{u}_D)\|_{L^2(\Omega)}$,
\begin{align*}
  \|\na(\widehat{u}-\widehat{u}_D)\|_{L^2(\Omega)}
  \le \exp(2\|V_e\|_{L^\infty(\Omega)})\|\widehat{u}_D\|_{L^2(\Omega)}.
\end{align*}
Poincar\'e's inequality then shows the first estimate.
}

For the second estimate, we use the test function $\widehat{u}_1-\widehat{u}_2$ in the weak formulation of $\diver(e^{V_e^1}\na \widehat{u}_1-e^{V_e^2}\na\widehat{u}_2)=0$ and apply H\"older's inequality:
\begin{align*}
  \int_\Omega e^{V_e^1}|\na(\widehat{u}_1-\widehat{u}_2)|^2\dd x
  &= \int_\Omega(e^{V_e^1}-e^{V_e^2})\na\widehat{u}_2\cdot
  \na(\widehat{u}_1-\widehat{u}_2)\dd x \\
  &\le \|e^{V_e^1}-e^{V_e^2}\|_{L^\infty(\Omega)}
  \|\na\widehat{u}_2\|_{L^2(\Omega)}
  \|\na(\widehat{u}_1-\widehat{u}_2)\|_{L^2(\Omega)}.
\end{align*}
Consequently, by the mean-value theorem,
\begin{align*}
  \exp&(-\|V_e^1\|_{L^\infty(\Omega)})
  \|\na(\widehat{u}_1-\widehat{u}_2)\|_{L^2(\Omega)}
  \le \|e^{V_e^1}-e^{V_e^2}\|_{L^\infty(\Omega)}
  \|\na\widehat{u}_2\|_{L^2(\Omega)} \\
  &\le \exp\big(\max\{\|V_e^1\|_{L^\infty(\Omega)},
  \|V_e^2\|_{L^\infty(\Omega)}\}\big)\|V_e^1-V_e^2\|_{L^\infty(\Omega)}
  \|\na\widehat{u}_2\|_{L^2(\Omega)}.
\end{align*}
We know from the first part of the proof that $\|\na\widehat{u}_2\|_{L^2(\Omega)}\le C\exp({2}\|V_e^2\|_{L^\infty(\Omega)})$. Hence,
\begin{align*}
  \|\na(\widehat{u}_1-\widehat{u}_2)\|_{L^2(\Omega)}
  \le C\exp\big({4}\max\{\|V_e^1\|_{L^\infty(\Omega)},
  \|V_e^2\|_{L^\infty(\Omega)}\}\big)\|V_e^1-V_e^2\|_{L^\infty(\Omega)},
\end{align*}
and we conclude the proof after an application of Poincar\'e's inequality.
\end{proof}


\section{Preconditioned Crank--Nicolson Markov chain Monte--Carlo method}
\label{sec.MCMC}

We reconstruct the doping profile from pointwise noisy measurements of the current density using Bayesian inversion. For this, we use a high-dimensional Markov chain Monte--Carlo (MCMC) method to estimate the mean posterior of the unknown function $\gamma=\exp(V_e)$ and then the doping profile from $\gamma$ by equation \eqref{2.C}. 

\subsection{Markov chain Monte--Carlo methods}

In Bayesian inversion, the quantity of interest is uncertain. Thus, the goal is to reduce the uncertainty by updating the prior probability distribution of the unknowns with the help of available observations.  According to the Bayes theorem, this is done by conditioning the prior distribution on the data, which leads to the posterior distribution of the unknown parameters. Usually, sampling directly from the posterior distribution is not possible especially in high dimensions, and therefore, numerical methods like the Metropolis--Hastings MCMC algorithm are used. The idea is to define a Markov chain over possible values in such a way that the stationary distribution of the Markov chain yields the desired distribution. The construction is iterative, and in each iteration of the Markov chain, a new sample, according to a proposal distribution (e.g.\ a random walk), is proposed. Then, by comparing the likelihood of the new sample to that one of the old sample, the algorithm decides to reject or accept the new sample to leverage the parameter inference. The acceptance probability of the proposed states in the standard Metropolis--Hastings algorithm is defined by
\begin{align}\label{3.alpha}
  \alpha(\widetilde{V}_e^{(i+1)},V_e^{(i)})
  = \min\bigg\{1,\frac{\pi(y|\widetilde{V}_e^{(i+1)})
  \pi_0(\widetilde{V}_e^{(i+1)})q(V_e^{(i)}|\widetilde{V}_e^{(i+1)})}{
  \pi(y|V_e^{(i)})\pi_0(V_e^{(i)})q(\widetilde{V}_e^{(i+1)}|
  V_e^{(i)})}\bigg\},
\end{align}
where the proposed parameter in the basic random-walk Metropolis--Hastings algorithm is given by
\begin{align*}
  \widetilde{V}_e^{(i+1)} = V_e^{(i)} + \beta\xi^{(i)}, \quad
  \xi^{(i)}\sim\mathcal{N}(0,\mathbb{I}),
\end{align*}
where $\mathbb{I}$ is the unit matrix, and the proposal distribution satisfies $q(\widetilde{V}_e^{(i+1)}|V_e^{(i)})\sim\mathcal{N}(V_e^{(i)},$ $\beta^2\mathbb{I})$. This distribution is symmetric in the sense
\begin{align*}
  q(\widetilde{V}_e^{(i+1)}|V_e^{(i)}) = q(V_e^{(i)}|
  \widetilde{V}_e^{(i+1)}),
\end{align*}
which simplifies expression \eqref{3.alpha} of the acceptance probability. 

The parameter $\beta>0$ is the step size for the parameter space and needs to be tuned. If $\beta$ is too small, the parameter space is not explored efficiently, while if it is too large, the proposed states are mostly rejected. Both cases lead to large asymptotic variances and therefore to a smaller accuracy of the sampling. According to \cite[Theorem 3.1]{GRG96}, there is a heuristic strategy to find the approximately optimal Metropolis algorithm, in the sense that the scale of the jumping (proposal) density is chosen in such a way that the average acceptance rate of the algorithm is roughly 0.234. This leads to a small asymptotic variance, and to this end, the parameter $\beta$ should be proportional to the inverse of the dimension of the parameter space (see \cite[Sec.~1]{RGG97} or \cite[Sec.~6.3]{CRSW13}). The challenge here is to find a suitable tuning parameter in the infinite-dimensional setting, as the acceptance probability tends to zero and thus, the asymptotic variance of the samples goes to infinity \cite[Sec.~1.1]{CRSW13}.

To tackle this issue, one way is to design the MCMC algorithm with well-defined proposal distributions such that the asymptotic variance of the generated samples is independent of the dimension of the parameter space. This is realized by the preconditioned Crank--Nicolson MCMC method for functions, which, in contrast to the random-walk Metropolis--Hastings algorithm, does not depend on the dimension of the parameter space.

The idea of the preconditioned Crank--Nicolson MCMC method is to make it applicable for functions. This is realized by using proposal distributions based on time-discretizations of stochastic dynamical systems, which preserve the Gaussian reference measure \cite{CRSW13}. The choice of the proposal density is important for two reasons: First, it should be easy to sample from it and second, it should balance the sampling accuracy and cost. Accuracy in sampling requires a small asymptotic variance and a reduction of correlation among the samples, since this leads to a larger effective sample size. On the other hand, sampling costs depend on the choice of the proposal density, and one could involve the gradient information of probability densities as well, like it is done in the Metropolis-Adjusted Langevin Algorithm \cite{Bes94}. 

The preconditioned Crank--Nicolson algorithm is an approach to construct the proposal function in the MCMC method. This proposal is independent of the dimension of the parameter space and therefore suitable in an infinite-dimensional setting. Furthermore, the preconditioned Crank--Nicolson algorithm is $\pi_0$-reversible, i.e.\ $\pi_0(\widetilde{V}_e^{(i+1)})q(V_e^{(i)}|\widetilde{V}_e^{(i+1)})
= \pi_0(V_e^{(i)})q(\widetilde{V}_e^{(i+1)}|V_e^{(i)})$. This implies, because of \eqref{3.alpha}, that
\begin{align*}
  \alpha(\widetilde{V}_e^{(i+1)},V_e^{(i)})
  = \min\big\{1,\exp\big(\Phi(V_e^{(i)},y)
  - \Phi(\widetilde{V}_e^{(i+1)})\big)\big\}
  = \min\bigg\{1,\frac{\pi(y|\widetilde{V}_e^{(i+1)})}{
  \pi(y|V_e^{(i)})}\bigg\},
\end{align*}
which means that the acceptance probability only depends on the likelihood function \cite{CRSW13}. Note that the acceptance probability is different from that one of the random-walk proposal and that the acceptance probability of the preconditioned Crank--Nicolson proposal is simpler and likelihood-informed. Moreover, the proposed state is accepted with probability (almost) one if the log-likelihood function $\Phi$ (almost) vanishes. Algorithm \ref{algo2} displays a pseudocode of the preconditioned Crank--Nicolson MCMC algorithm. 

\begin{algorithm}
\caption{Preconditioned Crank--Nicolson MCMC algorithm}
\label{algo2}
\begin{itemize}
\item{\bf Input:} prior measure $\mu_0(V_e)$ with $V_e\sim\mathcal{N}(m_0,\Sigma_0)$; likelihood function $\pi(y|V_e)$.
\item {\bf Output:} samples $V_e^{(i)}$ from the posterior measure $\mu^y$.
\end{itemize}
\begin{enumerate}
\item[(i)] Set $i=0$ and draw $V_e^{(0)}$ from prior, $V_e^{(0)}\sim\mathcal{N}(m_0,\Sigma_0)$.
\item[(ii)] For $i=1,\ldots,N$:
\begin{enumerate}
\item[(a)] Propose 
\begin{align*}
  \widetilde{V}_e^{(i+1)}
  = \sqrt{1-\beta^2}(V_e^{(i)}-m_0) + \beta\xi^{(i)} + m_0
  \quad\mbox{with }\xi^{(i)}\sim\mathcal{N}(0,\Sigma_0)
\end{align*}
according to the proposal distribution
\begin{align*}
  q(\widetilde{V}_e^{(i+1)}|V_e^{(i)})
  = \mathcal{N}\big(\sqrt{1-\beta^2}(V_e^{(i)}-m_0) 
  + m_0,\beta^2\Sigma_0\big).
\end{align*}
\item[(b)] Compute the acceptance probability
\begin{align*}
  \alpha:=\alpha(\widetilde{V}_e^{(i+1)},V_e^{(i)})
  = \min\bigg\{1,\frac{\pi(y|\widetilde{V}_e^{(i+1)})}{
  \pi(y|V_e^{(i)})}\bigg\}
\end{align*}
with the log-likelihood function \eqref{3.like}.
\item[(c)] Draw $r\sim U(0,1)$; if $r\le\alpha$ then $V_e^{(i+1)}:=\widetilde{V}_e^{(i+1)}$, else $V_e^{(i+1)}:=V_e^{(i)}$.
\end{enumerate}
\end{enumerate}
\end{algorithm}


\subsection{Infinite-dimensional setting: prior sampling}\label{sec.prior}

In the semiconductor inverse problem, the parameter space $X=L^2(\Omega)$ with $\Omega\subset\R^2$ is infinite dimensional. We discretize it by $n=n_x\times n_y$ parameters, where $n_x$ and $n_y$ are the numbers of parameters in the $x$ and $y$ direction, respectively. 

As mentioned in Section \ref{sec.lpp}, we assume a Gaussian prior random field. The spectral theorem for self-adjoint operators on Hilbert spaces provides a decomposition of Gaussian measures in Hilbert spaces, the so-called Karhunen--Lo\`eve expansion. The Karhunen--Lo\`eve theorem states that that a stochastic process can be represented as an infinite linear combination of orthogonal functions. The coefficients are stochastically independent standard random variables, and the orthogonal basis functions are determined by the spectral decomposition of the covariance function of the stochastic process. The Karhunen--Lo\`eve expansion of the Gaussian random field $V_e\sim\mathcal{N}(m_0,\Sigma_0)$ is defined by
\begin{align}\label{3.KL}
  V_e(x)= m_0(x) + \sum_{i=1}^{\infty} \sqrt{\lambda_i}\xi_i\phi_i(x),
  \quad \xi_i\sim\mathcal{N}(0,\mathbb{I}),
\end{align}
where $(\lambda_i,\phi_i)_{i=1}^\infty$ are the eigenpairs of the covariance operator with kernel function $\Sigma_0$, given by \eqref{3.MW}.

By discretizing the prior random field, the corresponding covariance function is reduced to a covariance matrix of order $n\times n$. For the sampling from the Gaussian prior random field in the semiconductor inverse problem, we use a truncated Karhunen--Lo\`eve expansion of $V_e$ with $n$ terms. This truncation yields an approximation of $V_e$ in terms of $n$ standard normal random variables, thus reducing the infinite-dimensional to a finite-dimensional stochastic parameter space.

As mentioned in Section \ref{sec.lpp}, we choose the Mat\'ern--Whittle covariance operator for $\Sigma_0$. It represents the covariance between two measurements as a function of the distance between the points at which they are taken. Furthermore, a Mat\'ern-type operator is stationary as it only depends on distances between points, and it is isotropic as the distance is Euclidean. 


\section{Numerical results}\label{sec.num}

The semiconductor domain is chosen as the square $\Omega=(-1,1)^2\subset\R^2$ with two Ohmic contacts; see Figure \ref{fig.diode}. The forward model is defined by the stationary linearized unipolar drift-diffusion equations \eqref{2.u1}--\eqref{2.Ve}.  The true (scaled) doping profile is the piecewise constant function
\begin{align*}
  C(x) = \begin{cases}
  -2 &\quad\mbox{for }x\in\Omega_P, \\
  1 &\quad\mbox{for }x\in\Omega_N,
  \end{cases}
\end{align*}
the (scaled) built-in potential is chosen as $V_{\text{bi}}=0.6$, and we take $\delta=\lambda=1$. For a given doping profile, we first solve the Poisson equation \eqref{2.Ve} and then the continuity equation \eqref{2.u1}--\eqref{2.u2}, using the applied voltage $U=2$ on the contact $\Gamma_P$. We have chosen continuous piecewise linear finite elements on a mesh of $20\times 20$ squares for the numerical discretization of the PDEs. Figure \ref{fig1} illustrates the true doping profile and the solutions to \eqref{2.Ve} and \eqref{2.u1}--\eqref{2.u2}. For the likelihood function, which is needed to compute the acceptance probability in Algorithm 2, we have chosen 1\% noise, i.e.\ $\Sigma=0.01\cdot\mathbb{I}$ (see \eqref{3.obmodel}).

\begin{figure}[ht]
\centering
\includegraphics[width=53mm,height=50mm]{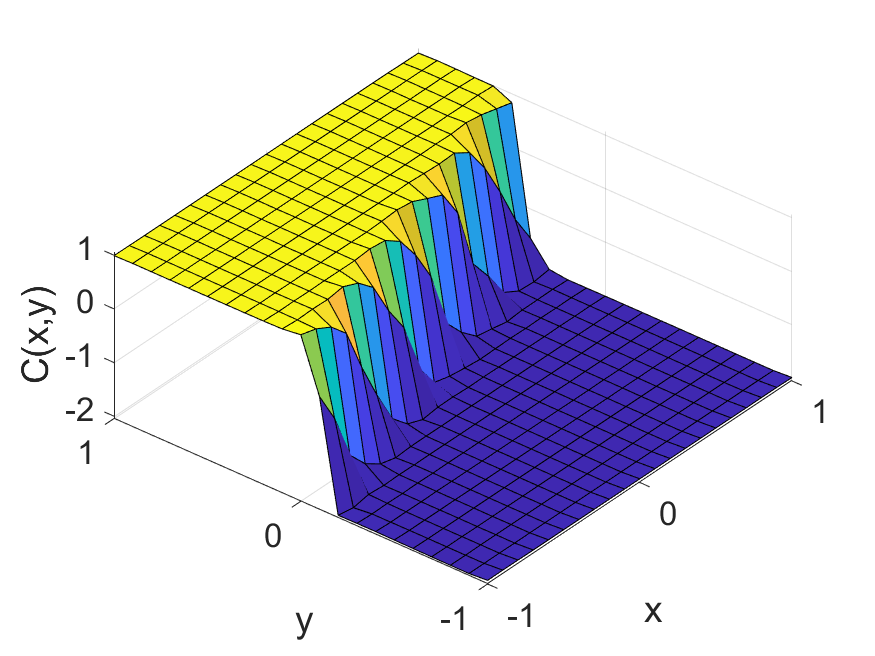} 
\includegraphics[width=53mm,height=50mm]{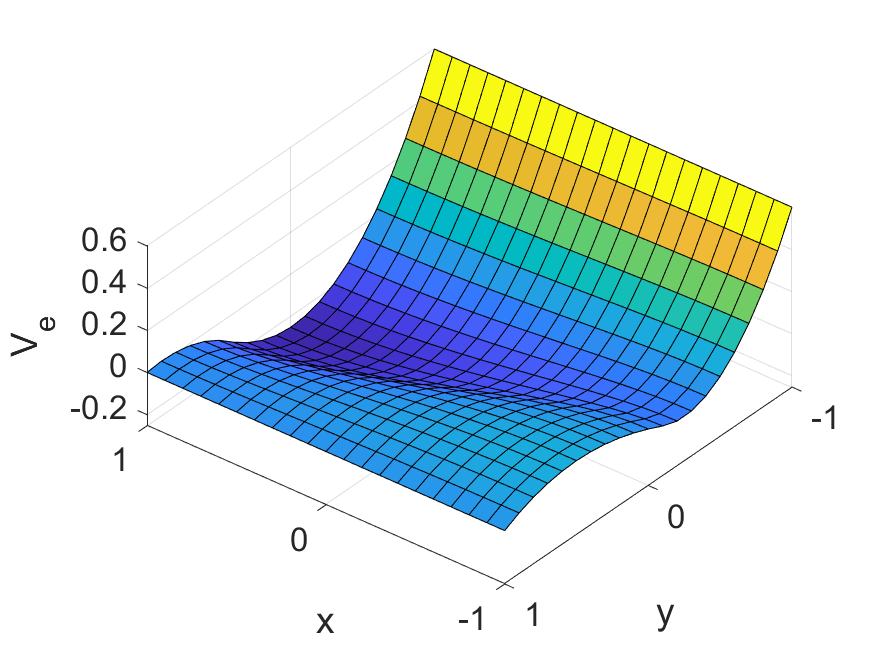}
\includegraphics[width=53mm,height=50mm]{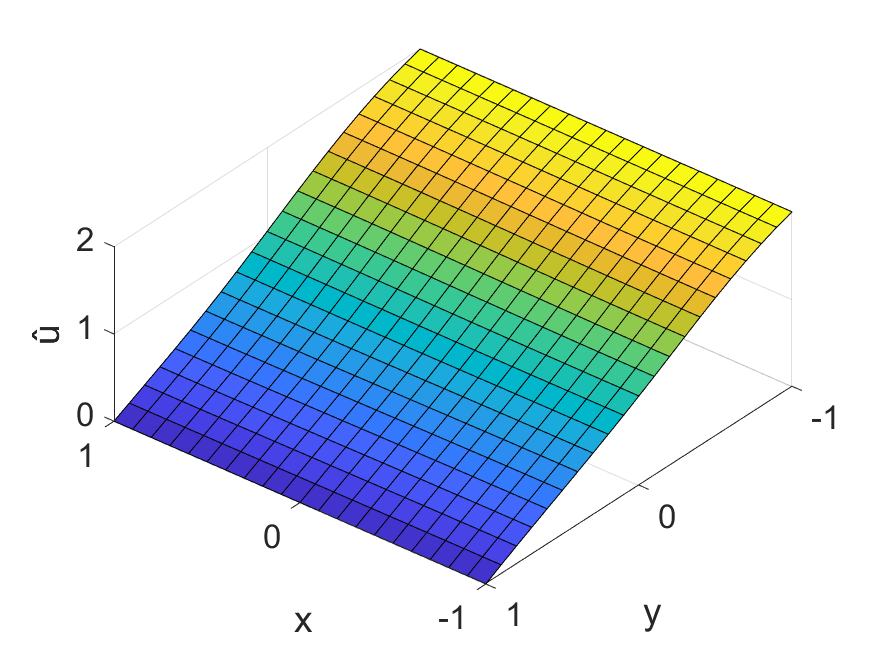}
\caption{True doping profile (left), solution $V_e$ to the Poisson equation (middle), and solution $\widehat{u}$ to the continuity equation (right).}
\label{fig1}
\end{figure}

We reconstruct the doping profile from the current density measurements as follows. First, we estimate the posterior distribution of $V_e$ using Algorithm \ref{algo2}. To this end, we compute the synthetic truth from the Poisson equation using the true doping profile and $V_{\text{bi}}=0.6$. Figure \ref{fig2} (left) shows the true parameter field $V_e$ with the location of 21 current density measurement points on the top contact $\Gamma_N$ (black dots). Here, we assume a Gaussian prior field with the Mat\'ern--Whittle covariance operator \eqref{3.MW}, where the parameters are $\sigma^2 = 0.01$, $\nu=1$, $\lambda=0.7$. The prior mean $m_0$ is taken as a perturbation of the true $V_e$, and the prior covariance operator $\Sigma_0$ (see Algorithm 2) becomes after discretization a matrix of dimension $441\times 441$. Then we discretize the Gaussian random field by using the truncated Karhunen--Lo\'eve expansion (with $N_{KL}=441$ terms) from \eqref{3.KL} in the discretized parameter domain with $n=20\times 20$ parameters. Figure \ref{fig2} (middle and right) illustrate, respectively, a realization of the Gaussian prior random field and the reconstructed parameter field by the estimation of the posterior using the preconditioned Crank--Nicolson MCMC algorithm. Here, we assume a preconditioned Crank--Nicolson proposal with jumping step size $\beta=0.2$ and $10^5$ samples. The posterior path plots and the corresponding histogram of posterior probability distribution functions of three parameters are displayed in Figure \ref{fig3}.

\begin{figure}[ht]
\centering
\includegraphics[width=53mm,height=50mm]{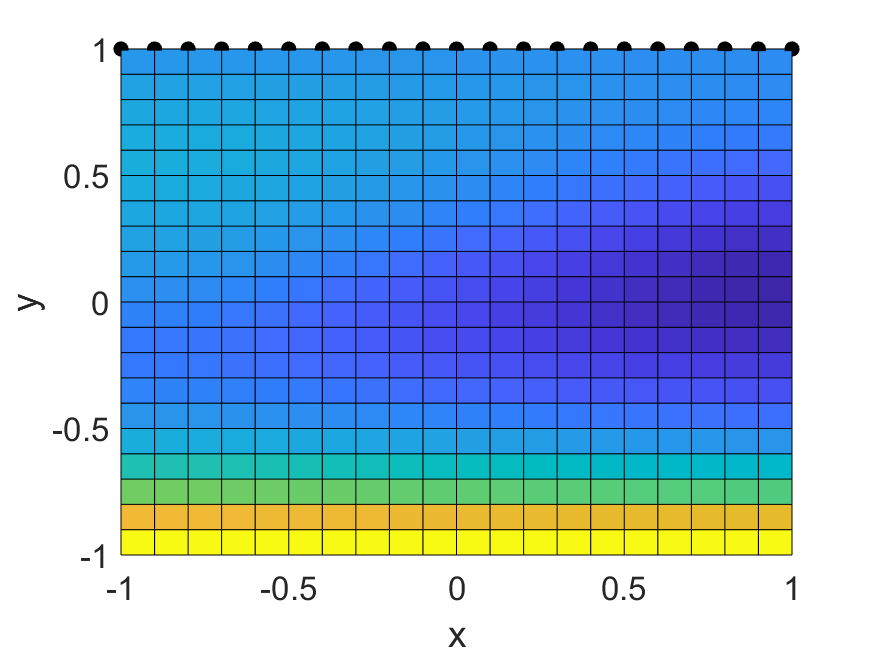}
\includegraphics[width=53mm,height=50mm]{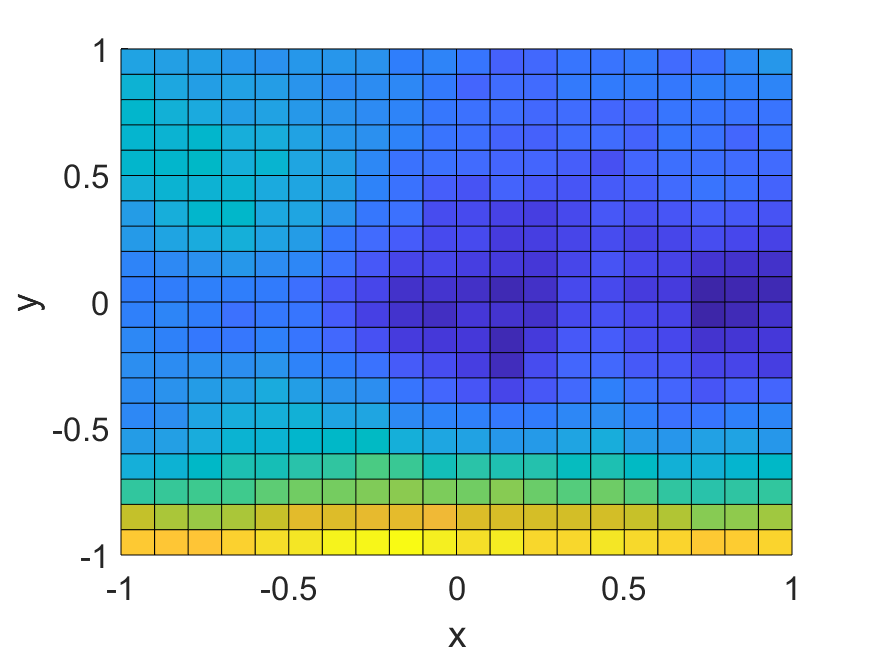}
\includegraphics[width=53mm,height=50mm]{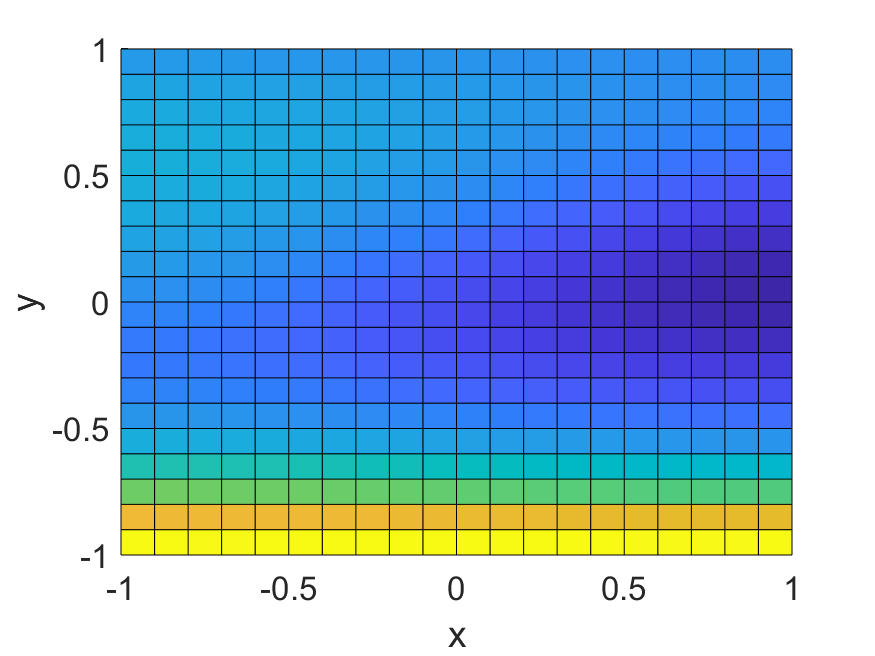}
\caption{True parameter field $V_e$ (synthetic truth) with the locations of the measurement points as black dots (left), a sample from the prior field (middle), and the reconstructed $V_e$ (right).}
\label{fig2}
\end{figure}

\begin{figure}[ht]
\centering	
\includegraphics[width=115mm,height=100mm]{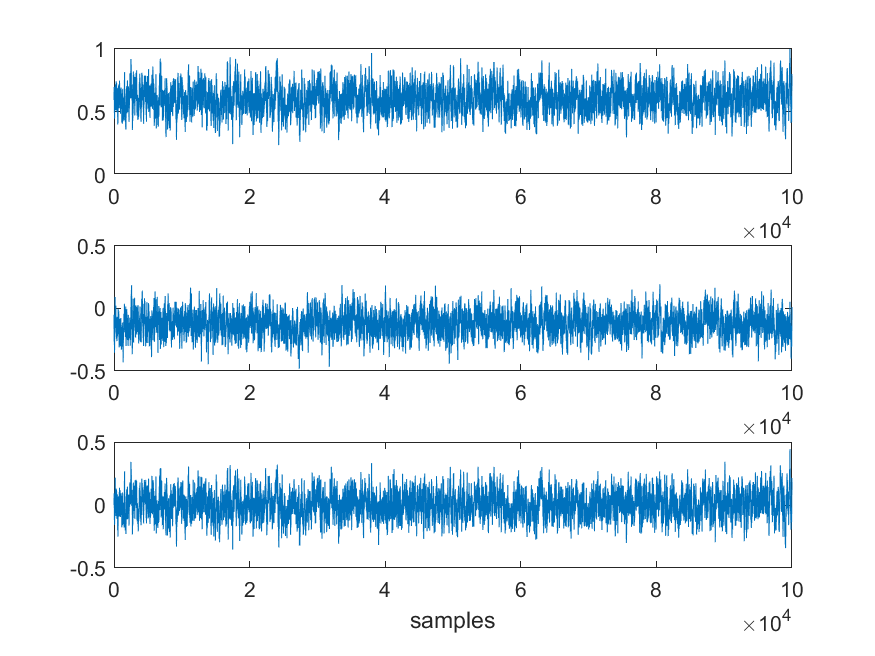}
\includegraphics[width=45mm,height=100mm]{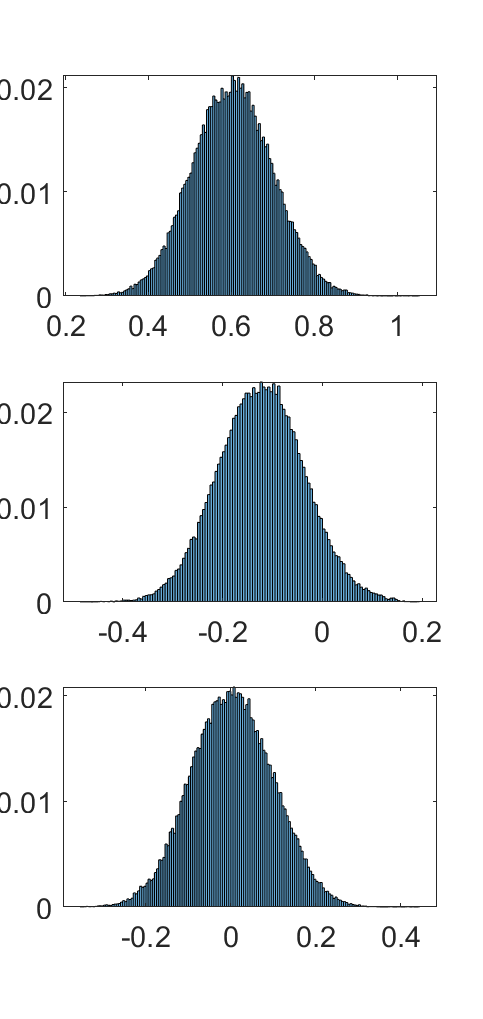}
\caption{The trace plots (left) and histograms (right) show the posterior Markov chains for $V_e$ for the first, 200th, and 400th parameters of the discretized parameter domain (from top to bottom).}
\label{fig3}
\end{figure}

When the parameter field $V_e$ is reconstructed, we can estimate the doping profile from equation \eqref{2.C}, where we use a finite-difference approximation of the Laplace differential operator. 
The doping profile reconstruction is reliable, since we use a physics-informed prior by solving the Poisson equation for some given doping profile. The PDE-constrained prior random field allows the preconditioned Crank--Nicolson MCMC method to beat the curse of dimensionality as much as possible \cite{CRSW13}. 

In Figure \ref{fig4}, we illustrate the true doping profile together with three reconstructed profiles with various applied voltages, while the other parameters are the same as for the true field. For larger applied voltages, the reconstructed doping function ``oscillates'', while the estimation of the function is more accurate for a small applied voltages $U = 2$, which agrees with the theory of semiconductor equations.

Furthermore, we test the Bayesian approach for different hyperparameters including different numbers of truncated Karhunen--Lo\'eve expansion terms for prior sampling as well as for different finite-element mesh sizes. The results are shown in Figures~\ref{fig5} and \ref{fig6}. The use of more terms in the truncated Karhunen--Lo\'eve expansion as well as a finer mesh for the finite-element solution of the forward model (i.e.\ more accurate PDE solution) leads to a more accurate Bayesian extraction of the doping profile.  However, the elapsed time for the Bayesian method on the finer mesh (with maximum diameter $h_{\text{max}} = 0.1$ for the triangle elements of the mesh) is 1.7 times longer than that on the coarser mesh ($h_{\text{max}} = 0.2$). We calculated the mean-square error (MSE) of the reconstructed and the true doping profiles for different applied voltages and $N_{KL}=441$, which are shown in Table~\ref{tbl1}.
\begin{table}[thpb!]
\centering
\caption{Mean-square errors calculated for the doping profile, reconstructed using  the Bayesian approach and finite-element method with mesh size $h_{\text{max}} = 0.1$ for different parameter values.}
\scalebox{1}{
	\begin{tabular}{ |c|c|c|c|} 
	\hline
	
    {MSE} & $U = 2$ & $U = 5$ & $U = 10$  \\
   \hline\hline
    $N_{KL} = 441$ &0.0751 &0.1046&0.2014\\
\cline{2-4}
\cline{2-4} 
\cline{2-4}
\hline
\cline{1-4}
\hline
\end{tabular}}\label{tbl1}
\end{table}
As shown in this table, for larger applied voltages, the MSE is larger. We note that if the number of truncated Karhunen--Lo\'eve expansion terms is small (e.g $N_{KL}=100$), the MSE increases ($\mbox{MSE}=0.1108$, where $U=2$). As mentioned, applying a coarser finite-element mesh leads to a worse MSE but the computational cost is smaller ($\mbox{MSE} = 0.0845$ for $h_{\text{max}} = 0.2$, where $U=2$).

The computation time (elapsed time) for the Bayesian algorithm with $10^5$ samples and using the finite-element method with mesh size $h_{\text{max}} = 0.1$ as the forward solver is about $80$ minutes (using MATLAB R2023a on an AMD Ryzen 5 5600G 6-core processor with 13.5 GB main memory), where more than $90\%$ of the total time is spent on the likelihood evaluation. 
The largest part of the used memory is allocated to store the posterior chain, which is an array of dimension $441\times 10^5$ and takes about $336$~MB memory. The data storage for the PDE evaluation takes about $16$~MB memory, and the memory allocated to the data storage for the prior and proposal covariance calculations, which are arrays of dimension $441\times 441$, takes about $1.5$~MB memory each. 

According to the numerical results, we obtain an optimal reconstruction of the doping profile with the applied voltage $U = 2$, $N_{KL}=441$ Karhunen--Lo\'eve expansion terms, and with a finite-element solver on a finer mesh. Therefore, tuning these design parameters are of importance for an accurate Bayesian estimation in semiconductor devices. This can be done by a Bayesian approach or optimal Bayesian experimental design (see \cite{APSG16}), which is future work.

\begin{figure}[htb!]
\centering
\includegraphics[width=70mm,height=60mm]{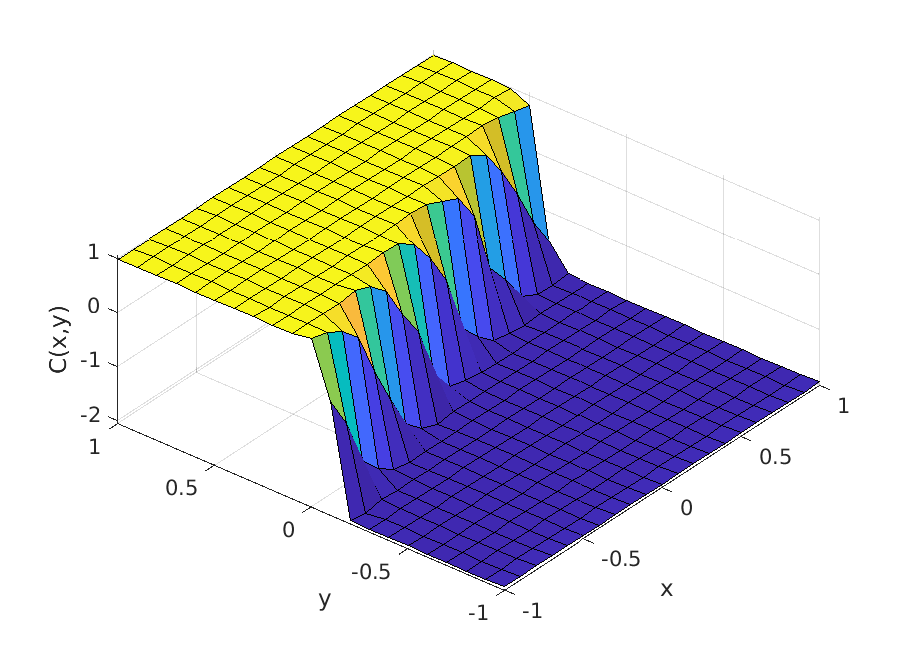} 
\includegraphics[width=70mm,height=60mm]{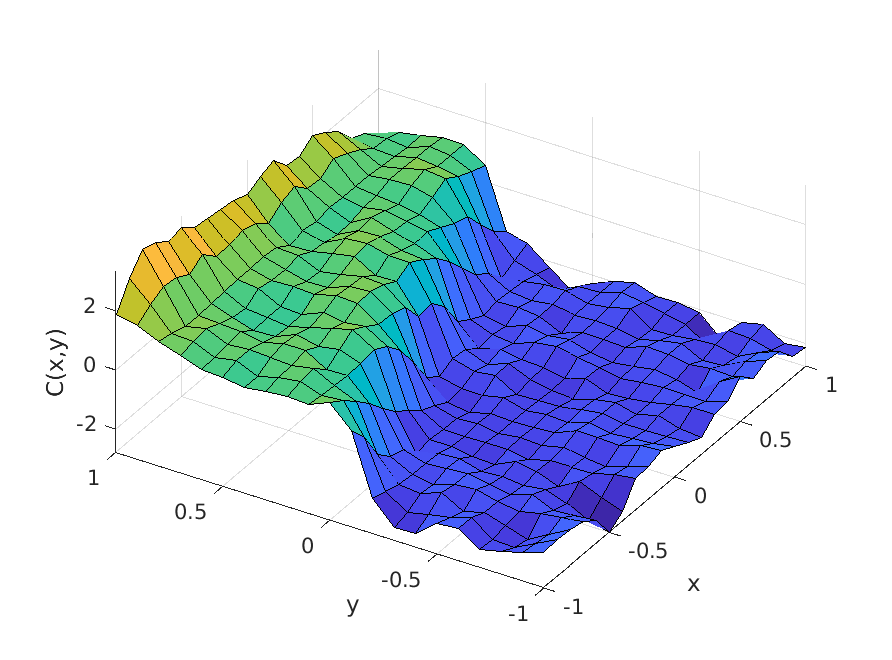}\vskip-5mm
\includegraphics[width=70mm,height=60mm]{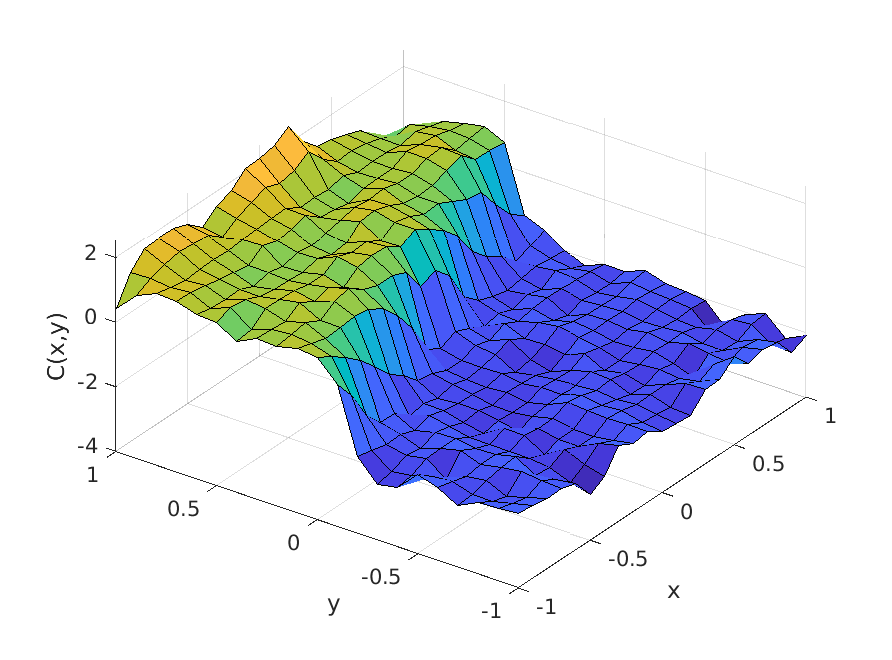}
\includegraphics[width=70mm,height=60mm]{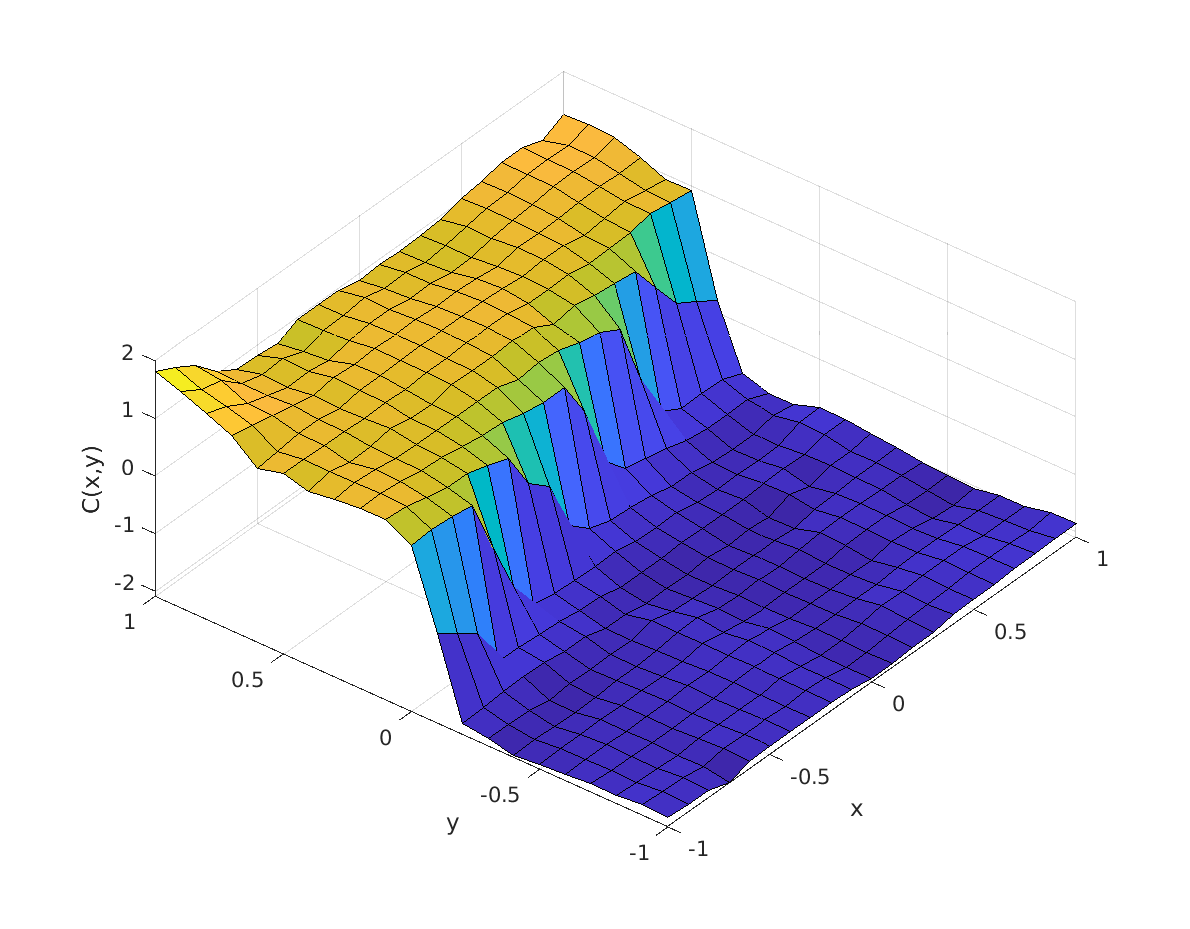}
\caption{The true (top-left) and reconstructed doping profiles using the preconditioned Crank--Nicolson MCMC method with applied voltage $U = 10$ (top-right), $U = 5$ (bottom-left), and $U = 2$ (bottom-right). }
\label{fig4}
\end{figure}

\begin{figure}[htb!]
\centering
\includegraphics[width=55mm,height=45mm]{Figures/TrueCdop-3D-noise0.01-noBar}\hskip-3mm
\includegraphics[width=55mm,height=45mm]{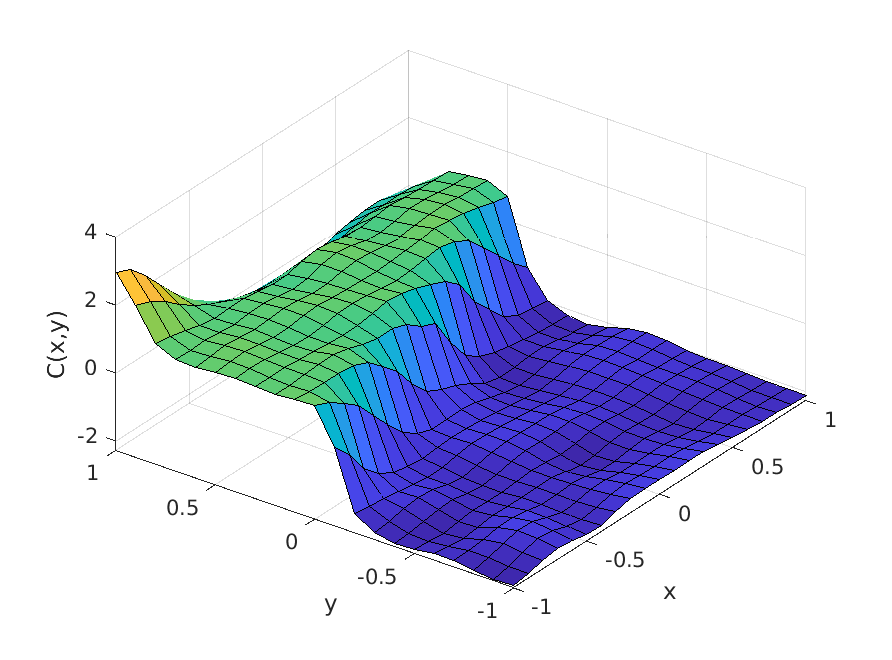}\hskip-3mm
\includegraphics[width=55mm,height=45mm]{Figures/Reconstructed-Cdop-3D-noise0.01-noBar}
\caption{The true (left) and reconstructed doping profiles using the preconditioned Crank--Nicolson MCMC method with prior sampling with $100$ (middle), and $441$ (right) truncated Karhunen--Lo\'eve expansion terms. }
\label{fig5}
\end{figure}

\begin{figure}[htb!]
\centering
\includegraphics[width=55mm,height=45mm]{Figures/TrueCdop-3D-noise0.01-noBar} \hskip-3mm
\includegraphics[width=55mm,height=45mm]{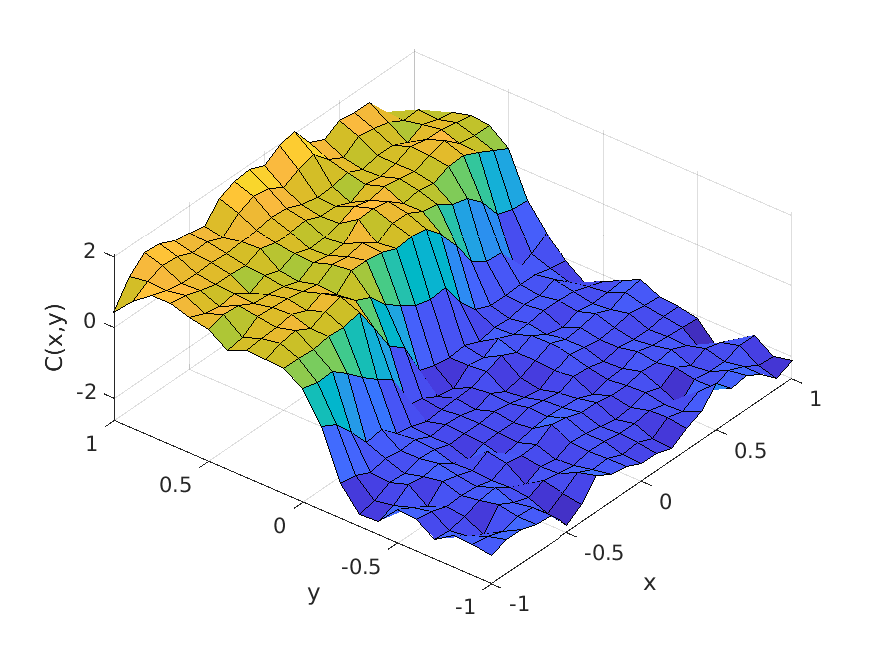}\hskip-3mm
\includegraphics[width=55mm,height=45mm]{Figures/Reconstructed-Cdop-3D-noise0.01-noBar}
\caption{The true (left) and reconstructed doping profiles using the preconditioned Crank--Nicolson MCMC method with a finite-element solver on a mesh with a maximum edge length of $0.2$ (middle), and $0.1$ (right).}
\label{fig6}
\end{figure}

\section{Conclusions}\label{sec.conc}

We have formulated a doping profile inverse problem for semiconductors as a Bayesian inverse problem. Under simplifying assumptions on the physical model, the forward model consists of the stationary linearized unipolar drift-diffusion equations. This leads to a reduced Bayesian inverse problem, consisting of an inverse problem for the exponential of the electric potential in the continuity equation and an explicit source inversion of the Poisson equation. The advantage is that the solution of the Poisson equation, given a doping profile, provides a physics-informed prior knowledge about the inversion parameter field for the main Bayesian inverse problem in the continuity equation. This allows for an accurate posterior estimation and Bayesian reconstruction of the doping profile. 

The extension of the doping inverse problem to a bipolar model-based forward problem is straightforward. To this end, the current density measurements are calculated by summing the electron and hole current densities, i.e.\ $(\widehat{J}_n +\widehat{J}_p)\cdot\nu|_{\Gamma_N}$, and the doping profile is obtained from $C(x)=\delta^2(\gamma(x)- \gamma(x)^{-1}) - \lambda^2\Delta(\ln\gamma)(x)$. The parameter $\gamma$ is reconstructed from two continuity equations for electrons and holes, given their concentrations. 

The Bayesian inversion problem can also be extended to the drift-diffusion equations for memristors, which consist of three continuity equations for the electron, hole, and oxide vacancy densities. While the analysis of the instationary equations is rather well established \cite{JJZ23,JuVe23}, many questions for the stationary equations are still open, for instance the uniqueness of weak solutions, their stability, and the Lipschitz continuity with respect to the data. The reason is the lack of sufficient regularity of the solutions. This issue may be overcome by considering the two-dimensional case only or by considering the degenerate equations \cite{JuVe23}, valid in high-injection regimes. In both cases, the regularity is improved, which may allow us to prove the well-posedness of the Bayesian inverse problem. This is work in preparation.



\begin{thebibliography}{11}

\bibitem{BaCa84} C.~Baiocchi and A.~Capelo. {\em Variational and Quasivariational Inequalities}. Wiley, Chichester, 1984.

\bibitem{Bes94} J.~Besag. Comments on ``Representations of knowledge in complex systems'' by U.~Grenander and M.~Miller''. {\em J. Royal Statist. Soc. B} 56 (1994), 591--592.

\bibitem{BELM04} M.~Burger, H.~Engl, A.~Leit\~ao, and P.~Markowich. On inverse problems for semiconductor equations. {\em Milan J. Math.} 72 (2004), 273--313.

\bibitem{BEMP01} M.~Burger, H.~Engl, P.~Markowich, and P.~Pietra. Identification of doping profiles in semiconductor devices. {\em Inverse Probl.} 17 (2001), 1765--1795.

\bibitem{CGR11} Y.~Cheng, I.~Gamba, and K.~Ren. Recovering doping profile in semiconductor devices with the Boltzmann--Poisson model. {\em J. Comput. Phys.} 230 (2011), 3391--3412.

\bibitem{CRSW13} S.~Cotter, G.~Roberts, A.~Stuart, and D.~White. MCMC methods for functions: Modifying old algorithms to make them faster. {\em Statist. Sci.} 28 (2013), 424--446.

\bibitem{FHL23} C.~Fai, C.~Hages, and A.~Ladd. Rapid optoelectronic characterization of semiconductors by combining Bayesian inference with Metropolis sampling. {\em PRX Energy} 2 (2023), no.~033013, 15 pages.

\bibitem{GRG96} A.~Gelman, G.~Roberts, and W.~Gilks. Efficient Metropolis jumping rules. {\em Bayesian Statist.} 5 (1996), 599--607.

\bibitem{HoNi17} B.~Hosseini and N.~Nigam. Well-posed Bayesian inverse problems: Priors with exponential tails. {\em SIAM/ASA J. Uncertain. Quantif.} 5 (2017), 436--465.

\bibitem{IbRo78} I.~Ibragimov and Y.~Rozanov. {\em Gaussian Random Processes}. Springer, New York, 1978.

\bibitem{JJZ23} C.~Jourdana, A.~J\"ungel, and N.~Zamponi. Three-species drift--diffusion models for memristors. {\em Math. Models Meth. Appl. Sci.} 33 (2023), 2113--2156.

\bibitem{Jue09} A.~J\"ungel. {\em Transport Equations for Semiconductors}. Lect.\ Notes Phys.\ 773. Springer, Berlin, 2009.

\bibitem{JuVe23} A.~J\"ungel and M.~Vetter. Degenerate drift--diffusion systems for memristors. Submitted for publication, 2023. arXiv:2311.16591

\bibitem{KTH21} A.~Karimi, L.~Taghizadeh, and C.~Heitzinger. Optimal Bayesian experimental design for electrical impedance tomography in medical imaging. {\em Comput. Meth. Appl. Mech. Engin.} 373 (2021), no.~113489, 17 pages.

\bibitem{KFBS95} N.~Khalil, J.~Faricelli, D.~Bell, and S.~Selberherr. The extraction of two-dimensional MOS transistor doping via inverse modeling. {\em IEEE Electron. Device Lett.} 16 (1995), 17--19.

\bibitem{KSH20} A.~Khodadadian, B.~Stadlbauer, and C.~Heitzinger. Bayesian inversion for nanowire field‑effect sensors. {\em J. Comput. Electron.} 19 (2020), 147--159.

\bibitem{Lei06} A.~Leit\~ao. Semiconductors and Dirichlet-to-Neumann maps. {\em Comput. Appl. Math.} 25 (2006), 187--203.

\bibitem{McL00} W.~McLean. {\em Strongly Elliptic Systems and Boundary Integral Equations}. Cambridge University Press, Cambridge, 2000.

\bibitem{Mar86} P.~Markowich. {\em Stationary Semiconductor Device Equations}. Springer, Wien, 1986.

\bibitem{MiHo92} S.~Middleman ad A.~Hochberg. {\em Process Engineering Analysis in Semiconductor Device Fabrication}. McGraw Hill, New York, 1992.

\bibitem{PFLRH24} S.~Piani, P.~Farrell, W.~Lei, N.~Rotundo, and L.~Heltai. Data-driven solutions of ill-posed inverse problems arising
from doping reconstruction in semiconductors. {\em Appl. Math. Sci. Engin.} 32 (2024), no.~2323626, 28 pages.

\bibitem{RGG97} G.~Roberts, A.~Gelman, and W.~Gilks. Weak convergence and optimal scaling of random walk Metropolis algorithms. {\em Ann. Appl. Probab.} 7 (1997), 110--120.


\bibitem{Sel84} S.~Selberherr. {\em Analysis and Simulation of Semiconductor Devices}. Springer, Wien, 1984.

\bibitem{SLM+18} S.~Somnath, K.~Law, A.~Morozovska, P.~Maksymovych, Y.~Kim, X.~Lu, M.~Alexe, R.~Archibald, S.~Kalinin, S.~Jesse, and R.~Vasudevan. Ultrafast current imaging by Bayesian inversion. {\em Nature Commun.} 9 (2018), no.~513, 11 pages.

\bibitem{Stu10} A.~Stuart. Inverse problems: A Bayesian perspective. {\em Acta Numer.} 19 (2010), 451--559.

\bibitem{TKPH20} L.~Taghizadeh, A.~Karimi, E.~Presterl, and C.~Heitzinger. Bayesian inversion for a biofilm model including quorum sensing. {\em Comput. Biol. Med.} 117 (2020), no.~103582, 11 pages.

\bibitem{Van00} A.~Van der Vaart. {\em Asymptotic Analysis}. Cambridge University Press, Cambridge, 2000.

\bibitem{APSG16} A.~Alexanderian, N.~Petra, G.~Stadler, and O.~Ghattas. A fast and scalable method for A-optimal design of experiments for infinite-dimensional Bayesian nonlinear inverse problems. {\em SIAM J. Sci. Comput.} 38 (2016), A243--A272.



\end{thebibliography}
\end{document}